\documentclass[11pt]{article}
\usepackage{amssymb,amsmath,amsthm}

\usepackage{url}

\def\A{{\cal A}}
\def\T{{\cal T}}
\def\P{{\cal P}}
\def\C{{\cal C}}
\def\RSK{\operatorname{RSK}}
\def\la{\lambda}

\def\bxi{\bar\xi}
\def\bet{\bar\eta}
\def\eps{\varepsilon}
\def\si{\sigma}
\def\sin{{\mathfrak S}_\infty}
\def\sn{{\mathfrak S}_n}
\def\shape{\operatorname{sh}}

\def\Y{{\mathbb Y}}
\def\Prob{\operatorname{Prob}}

\newtheorem{lemma}{Lemma}
\newtheorem{proposition}{Proposition}
\newtheorem{theorem}{Theorem}
\newtheorem{definition}{Definition}
\newtheorem{corollary}{Corollary}

\urldef\vershik\url{avershik@pdmi.ras.ru}
\urldef\natalia\url{natalia@pdmi.ras.ru}

\author{A.~M.~Vershik\thanks{%
St.~Petersburg Department of Steklov Institute of Mathematics; St.~Petersburg State University; Institute for Information Transmission Problems. E-mail: \vershik.}
\and N.~V.~Tsilevich\thanks{%
St.~Petersburg Department of Steklov Institute of Mathematics.
E-mail: \natalia.}}
\title{Ergodicity and totality of partitions associated with the RSK correspondence}

\date{}

\begin{document}
\maketitle

\begin{abstract}
We study asymptotic properties of sequences of partitions
($\sigma$\nobreakdash-algebras) in spaces with Bernoulli measures associated with the Robinson--Schensted--Knuth correspondence.

\end{abstract}

\section{Introduction}

We study dynamic properties of the Robinson--Schensted--Knuth correspondence (RSK) when it is successively applied to growing sequences of symbols. In particular, we are interested in asymptotic properties of this correspondence. Apparently, the ``dynamic'' approach to the RSK correspondence first appeared in~\cite{VK}, where it was discovered that applying the RSK algorithm   to an infinite sequence of independent symbols from a linearly ordered set defines a correspondence between Bernoulli measures (i.e., ensembles of independent sequences of symbols) and central (Markov) measures on the path space of a certain graph (the Young graph). In other words, the RSK correspondence defines measure-preserving homomorphisms from Bernoulli spaces to Markov path spaces. Thus, the question inevitably arises of whether or not this homomorphism is an isomorphism~$\bmod\,0$. The affirmative answer to this question was obtained in the relatively recent important papers~\cite{RS,Sniady}; their approach relies on the Sch\"utzenberger's
jeu de taquin and is rather technically involved.

The first author~\cite{21} suggested a program for solving these problems for a~certain class of graphs, including an elaboration of the ergodic method~\cite{V74} for finding invariant measures and the so-called {\it bernoullization} of graphs. This approach uses techniques of the theory of filtrations (decreasing sequences of  $\sigma$\nobreakdash-algebras)~\cite{filtr}; in particular, it has led to a new problem of characterization of  {\it de Finetti-like} filtrations, i.e., filtrations for which every ergodic central measure is a Bernoulli measure. Note in this regard that this paper (as well as~\cite{Sniady})  deals with the properties and structure of some ergodic central measures, those originating from Bernoulli measures; the fact that they exhaust all central measures with finitely many frequencies for the Young graph does by no means follow from these considerations. However, the methods suggested in~\cite{21} allowed the authors to prove, in a~paper in preparation, that every central measure of this type is associated with a~Bernoulli measure in the above sense. This is how a long awaited purely combinatorial proof of Thoma's theorem (see~\cite{Three}) should appear.

In this paper, we study only a part of the general problem in the simplest case of a linearly ordered set with finitely many symbols, and prove two facts in a~sense dual to each other: the totality of the coplactic (= dual Knuth) equivalence and the ergodicity of the so-called Young filtration, i.e., the tail filtration determined by the $Q$-tableaux in the RSK corresponcence. Our arguments, on the one hand, give another proof of the corresponding results from~\cite{Sniady} and, on the other hand, can be applied in a much more general situation. We consider separately the case of two letters, because it is illustrative and serves as the base case for an induction in the general case.

The reader is assumed to be familiar with the RSK correspondence and its basic properties (see, e.g.,~\cite{Fulton}); for background on the representation theory of the infinite symmetric group (the Young graph, central measures, Thoma parameters, etc.), see, e.g.,~\cite{Kerov}; for that on the theory of measurable partitions and filtrations, see, e.g.,~\cite{filtr}.

\section{Youngization} Let $\A=\{1,2,\ldots,k\}$ be a finite alphabet. Consider the space  $X=\A^\infty$ of infinite words in the alphabet~$\A$ with Bernoulli measure $m_p^\infty$, where $p=(p_1,p_2,\ldots,p_k)$, $p_i=\Prob(i)$, and $p_1\ge p_2\ge\ldots\ge p_k>0$.

Let $\T$ be the set of infinite standard Young tableaux (or, which is the same, the set of infinite paths in the Young graph~$\Y$) and $\mu_p$ be the central measure on~$\T$ with Thoma parameters $(p,0,0)$. Note that the measure~$\mu_p$ is supported by the subset of tableaux with at most $k$~rows.

By $\RSK(w)=(P(w),Q(w))$ we denote the result of applying the RSK algorithm to a finite sequence (word)~$w$ in the alphabet~$\A$. Thus, $P(w), Q(w)$ is a pair of Young tableaux of the same shape (with at most $k$~rows), which will be denoted by~$\shape(w)$; the tableau~$P(w)$ is semistandard, while the tableau~$Q(w)$ is standard. Given an infinite sequence $x\in X$, denote by $[x]_n=(x_1,\ldots,x_n)\in \A^n$ its initial segment of length~$n$, and let $\{x\}_{n+1}:=(x_{n+1},x_{n+2},\ldots)\in\A^\infty$ be its $(n+1)$-tail. Also, denote by~$P_n(x)$ and~$Q_n(x)$ the tableaux obtained by applying the RSK algorithm to the initial segment of length~$n$ of a sequence~$x$: $\RSK([x]_n)=(P_n(x),Q_n(x))$.

Following \cite{VK}, we introduce a map from the space of infinite sequences to the space of infinite Young tableaux.

\begin{definition}
{\rm Successively apply the RSK algorithm to the initial segments~$[x]_n$ of a sequence $x\in X$. It is clear from the construction of the algorithm that $\lim\limits_{n\to\infty}Q_n(x)=:Q(x)$ is an infinite standard Young tableau; denote it by~$\pi(x)$.
The resulting map
\begin{equation}\label{hom}
\pi:(X,m_p^\infty)\to(\T,\mu_p)
\end{equation}
is called the \emph{youngization}.}
\end{definition}

In~\cite{VK} it is proved that the youngization~\eqref{hom} is a homomorphism of measure spaces.

\section{The sequences of Young partitions on a Bernoulli space. The main theorems}

The following measurable partitions are defined in a natural way on the
space~$(X,m_p^\infty)$ of infinite Bernoulli sequences:
\begin{itemize}
\item the {\it cylinder} partition $\si_n$ of level~$n$, whose element is a set (of finite measure) of sequences with a fixed initial segment of length~$n$ and arbitrary tail;
\item the {\it tail} partition $\tau_n$ of level~$n$, whose element is a (finite) set of sequences with a fixed $(n+1)$-tail  and arbitrary beginning.
\end{itemize}
We will call them the {\it Bernoulli} cylinder and tail partitions. They have the following properties:
\begin{itemize}
\item the sequence of partitions $\si_n$ is monotonely increasing and converges in the weak topology to the partition~$\eps$ into separate points;
\item the sequence of partitions $\tau_n$ is monotonely decreasing and converges in the weak topology to the trivial partition~$\nu$;
\item for every $n$, the partitions $\tau_n$~and~$\si_n$ are independent with respect to the Bernoulli measure~$m_p^\infty$.
\end{itemize}

Now consider the following measurable partitions on the space~$\T$ of infinite Young tableaux:
\begin{itemize}
\item the {\it cylinder} partition $\xi_n$ of level~$n$, whose element is a set (of finite measure) of infinite paths in the Young graph (i.e., infinite Young tableaux) with a fixed initial segment of length~$n$ and arbitrary tail.
\item the {\it tail} partition $\eta_n$ of level~$n$, whose elements is a (finite) set of infinite paths in the Young graph with a fixed $n$-tail and  arbitrary beginning.
\end{itemize}

\begin{definition}
{\rm The {\it Young cylinder partition} and {\it Young tail partition} of the space~$X$ of infinite sequences are the partitions
 $\bxi_n:=\pi^{-1}\xi_n$ and ${\bet_n:=\pi^{-1}\eta_n}$, respectively, i.e., the preimages of the cylinder and tail partitions on Young tableaux under the youngization~$\pi$.
The decreasing sequence of partitions~$\{\bet_n\}$ will also be called the {\it Young filtration}.}
\end{definition}

Thus, $x\sim_{\bxi_n}y$ $\iff$ $Q([x]_n)=Q([y]_n)$, and $x\sim_{\bet_n}y$ $\iff$
$\shape([x]_N)=\shape([y]_N)$ for $N\ge n$. Obviously, $\bxi_n\prec\sigma_n$.

To begin with, we describe the structure of Young partitions. Recall (see, e.g.,~\cite{Fulton}) that the  {\it Knuth equivalence} (or {\it plactic}) {\it class}~$\P_t$ and the {\it dual Knuth equivalence} (or {\it coplactic}) {\it class}~$\C_t$ corresponding to a given Young tableau~$t$ of size~$n$ is the set of all words~$u$ of length~$n$ such that $P(u)=t$ and $Q(u)=t$, respectively.

\begin{theorem}\label{th:part}
The Young partitions on the space~$X$ can be described as follows.
\begin{itemize}
\item The elements of~$\bxi_n$ are indexed by the standard Young tableaux~$t$ of size~$n$ and coincide with the coplactic classes~$\C_t$.
\item The elements of~$\bet_n$ are indexed by the pairs~$(t,y)$ where $t$ is a semistandard Young tableau of size~$n$ and $y$~is an infinite word in the alphabet~$\A$ and have the form
$$\{x\in X: [x]_n\in\P_t,\,\{x\}_{n+1}=y\};
$$
in other words, this is the set of all sequences whose initial segment of length~$n$ belongs to a given plactic class and the tail coincides with a~given infinite sequence.
\end{itemize}
\end{theorem}

The first assertion of this theorem is obvious, and the second one will be proved in the next section (see Lemma~\ref{l:tails}).

Recall that a decreasing sequence of partitions (filtration) in a measure space is said to be {\it ergodic} if it converges in the weak topology to the trivial partition~$\nu$ (into a single nonempty set). In turn, an increasing sequence of partitions in a measure space is said to be {\it total} if it converges in the weak topology to the partition~$\eps$ into separate points. Thus, the ergodicity of a sequence of partitions means that there is no nonconstant measurable function that is constant on the elements of all partitions, while the totality of a sequence of partitions means that for almost all pairs~$x,y$ of different points, $x$~and~$y$ will eventually fall in different elements of partitions.

Clearly, the sequence of partitions~$\bxi_n$ is increasing, while the sequence of partitions~$\bet_n$  is decreasing. Our purpose is to study the limiting partitions $\bxi:=\lim\limits_{n\to\infty}\bxi_n$ and $\bet:=\lim\limits_{n\to\infty}\bet_n$, namely, to prove the following theorem.

\begin{theorem}\label{th:main}
The Young partitions on the space~$X$ of infinite Bernoulli sequences have the following properties:
\begin{itemize}
\item the sequence of partitions $\bxi_n$ is total;
\item the sequence of partitions $\bet_n$ is ergodic.
\end{itemize}
\end{theorem}

As a corollary, we obtain the result proved (for an arbitrary central measure) in~\cite{Sniady}.

\begin{corollary}\label{cor:iso}
The youngization map~{\rm\eqref{hom}} is an isomorphism of measure spaces between $(X,m_p^\infty)$ and $(\T,\mu_p)$.
\end{corollary}

Note that the space $(\T,\mu_p)$ of infinite paths in the Young graph with the central measure~$\mu_p$  can be identified with the space of trajectories of a~Markov  random walk on the ``Weyl chamber''
$${\cal W}_k=\{(x_1,\ldots,x_k):x_i\in\mathbb Z,\, x_1\ge x_2\ge\ldots\ge x_k\ge0\},
$$ where for a path $\T\ni t=(\la^{(1)},\la^{(2)},\ldots)$ we set $\la^{(n)}=(\la_1^{(n)},\ldots,\la_k^{(n)})\in{\cal W}_k$ (thus, at each step, one of the coordinates is increased by~$1$). So, the youngization map~{\rm\eqref{hom}} establishes an isomorphism between the space~$(X,m_p^\infty)$ of trajectories of a {\it Bernoulli process} and the space of trajectories of a~{\it Markov process}.  For example, in the case of $k=2$ and the uniform measure $p=(\frac12,\frac12)$, the transition probabilities of the Markov process are given by the following formula (see~\cite{Markov}): if $j=\la_1-\la_2$ is the difference of the row lengths of a diagram, then
$$
\Prob(j,j+1)=\frac{j+2}{2(j+1)},\qquad \Prob(j,j-1)=\frac{j}{2(j+1)}.
$$

\smallskip
We also introduce another family of partitions~$\zeta_n$. Namely, on the space~$X$ of infinite sequences there is a natural action of the infinite symmetric group~$\sin$ by permutations of elements. Denote by~$\zeta_n$ the partition of~$X$ into the orbits of the finite subgroup~${\sn\subset\sin}$. In other words, two sequences $x,y\in X$ belong to the same element of~$\zeta_n$ if and only if $\{x\}_{n+1}=\{y\}_{n+1}$ and  in $[x]_n,[y]_n$ all  elements occur with the same multiplicity. The partitions~$\zeta_n$ will be called the {\it de Finetti partitions}.
Note that $\lim\limits_{n\to\infty}\zeta_n=\nu$ by the Hewitt--Savage zero--one law.

\section{Proofs of the main theorems}
\subsection{The case $k=2$}\label{sec:2}
In this section, we analyze the case of the two-letter alphabet $\A_2=\{1,2\}$. Note that the space $X_2=\A_2^\infty$  with Bernoulli measure~$m^\infty$, where $m=(p_1,p_2)$,  can be naturally regarded as the space of trajectories of the random walk on the one-dimensional lattice with probability~$p_1$ of moving right and probability~$p_2$  of moving left.
Recall that we assume that $p_1\ge p_2>0$.

Consider a sequence from $\A_2^n$  as a word $w=x_1\ldots x_n$. Bracket every factor~$21$ in~$w$. The remaining letters constitute a subword~$w_1$ in~$w$.
Bracket every factor~$21$ in~$w_1$. We are left with a word~$w_2$. Continue the procedure  until we are left with a word of the form $w_k=1^a2^b=x_{i_1}\ldots x_{i_{a+b}}$ with $a,b\ge0$. The elements $x_{i_1},\ldots ,x_{i_{a+b}}$ of the sequence~$w$ will be called {\it free}, and all the other elements will be called {\it paired}. The number of brackets will be called the
{\it rank} of~$w$ and denoted by~$r(w)$.

Note that it follows from the properties of the random walk on the one-dimensional lattice with $p_1\ge p_2$
that a.e.\ sequence $x\in X_2$ has an initial segment with more $1$'s than $2$'s. This means that $x$~contains infinitely many free $1$'s and each $2$ gets paired in a sufficiently long initial segment.

The  following lemma is an obvious consequence of the RSK construction.

\begin{lemma}
Let $\shape([x]_n)=(\la_1,\la_2)$. Then $\la_2=r([x]_n)$. Namely, the second row of~$Q([x]_n)$ contains the indices of the free $1$'s, while its first row contains the indices of all the other elements.
\end{lemma}

Now we can obtain an explicit description of the partitions $\bxi_n$~and~$\bet_n$, which, in particular, implies Theorem~{\rm\ref{th:part}} in the two-letter case.

\begin{proposition}\label{prop:young2}
The Young partitions on the set $X_2=\A_2^\infty$  can be described as follows:
\begin{itemize}
\item $x\sim_{\bxi_n}y$ $\iff$   all paired coordinates in $[x]_n,[y]_n$ coincide;
\item $x\sim_{\bet_n}y$ $\iff$ $[x]_n,[y]_n$ have the same rank  and  $1$'s~and~$2$'s occur in them with the same multiplicity
 \textup(these two conditions amount to the condition that $P_n(x)=P_n(y)$\textup) and ${\{x\}_{n+1}=\{y\}_{n+1}}$.
\end{itemize}
\end{proposition}

\begin{proof}
The first assertion is obvious. To prove the second one, we first show that $\bet_n\succ\tau_n$. Let $x\sim_{\bet_n}y$. We must prove that $x\sim_{\tau_n}y$, i.e., $x_N=y_N$ for $N\ge n+1$. Assume the contrary  and let $m$~be the index of the first coordinate that differs in $x$~and~$y$. Without loss of generality,
$x_m=1$, $y_m=2$. But $y_m=2$ is a free element in~$[y]_m$, hence $x_m=1$~is a free element in~$[x]_m$ (otherwise,  $\shape([x]_m)\ne\shape([y]_m)$). Then the tail~$\{y\}_{m+1}$ contains no free $1$'s, which, as we have noted above, has probability~$0$. So, $\bet_n\succ\tau_n$. The coincidence of ranks is obvious. It remains to show that in $[x]_n$~and~$[y]_n$ the elements $1$~and~$2$ occur with the same multiplicitiy. Assume to the contrary that, say,  $[y]_n$~has more free $2$'s than~$[x]_n$. Since, almost surely, each $2$ becomes paired in a sufficiently long initial segment, at the moment when the ``extra'' $2$ gets paired, the condition $\shape([x]_N)=\shape([y]_N)$ fails.
\end{proof}

\begin{corollary}\label{cor:tails}
The three (Bernoulli, Young, and de Finetti) tail filtrations on $X_2$ satisfy the relation
$$
\tau_n\prec\zeta_n\prec\bet_n.
$$
\end{corollary}

\begin{proposition}\label{prop:k2}
For the two-letter alphabet, Theorem~{\rm\ref{th:main}} holds, i.e., the sequence of Young cylinder partitions is total and the Young filtration is ergodic.
\end{proposition}

\begin{proof}
Let $x\sim_{\bxi}y$ and $x\ne y$. Then there exists~$n$ such that $[x]_{n-1}=[y]_{n-1}$ and $x_n\ne y_n$. But $x\sim_{\bxi_n}y$, hence all paired coordinates in $[x]_n$~and~$[y]_n$ coincide, so only free ones may differ. Without loss of generality, let $x_n$~be a~free~$1$ and $y_n$~be a free~$2$.  Almost surely, there exists $N>n$ such that this~$2$ gets paired in~$[y]_N$. But the element
$x_n=1$ remains free in~$[x]_N$. Then  $x\nsim_{\bxi_N}y$, a contradiction. This proves that $\bxi=\eps$.

Now we prove that $\bet=\nu$. Consider the de Finetti partitions~$\zeta_n$; we will prove that if $x\sim_{\zeta_n}y$, then there exists~$N$ such that $x\sim_{\bet_N}y$. Since  $\lim\limits_{n\to\infty}\zeta_n=\nu$ by the Hewitt--Savage zero--one law, this implies that ${\lim\limits_{n\to\infty}\bet_n=\nu}$, as required.

So, let $x\sim_{\zeta_n}y$ but $x\nsim_{\bet_n}y$, i.e., $\{x\}_{n+1}=\{y\}_{n+1}=:z$,  the multiplicities of~$1$'s~and~$2$'s in $[x]_n,[y]_n$ coincide, but $r([x]_n)\ne r([y]_n)$. Consider the common tail~$z$ of $x$~and~$y$. As free $1$'s appear in~$z$, they get paired with free $2$'s in $[x]_n$~and~$[y]_n$. Let $N$~be the moment when the last of the free $2$'s in~$[x]_n,[y]_n$ gets paired. It is easy to see that $\shape([x]_N)=\shape([y]_N)$ and, consequently, $x\sim_{\bet_N}y$. As discussed above, this completes the proof.
\end{proof}

\subsection{The general case}\label{sec:gen}

In this section, we prove Theorems~\ref{th:part}~and~\ref{th:main} in full generality.
Recall that ${p_1\ge p_2\ge\ldots\ge p_k>0}$. We need the following lemma, which shows that, almost surely, each element
$a>1$  eventually gets bumped from the first row of the $P$-tableau.

\begin{lemma}\label{l:disappear}
Fix $\ell=2,\ldots,k$ and denote by $m_n=m_n(\ell,x)$ the number of elements equal to~$\ell$ in the first row of the tableau~$P_n(x)$ for a random sequence $x\in X$. Then for every $q\in{\mathbb N}$, almost surely, there exists $N\ge q$ such that $m_N=0$.
\end{lemma}
\begin{proof}
If $m_q=0$, there is nothing to prove. Let $m_q\ne0$. Denote by~$a_n$ the greatest element less than~$\ell$ in the first row of~$P_n(x)$ (or~$1$ if there is no such element).
Clearly,
$$
m_{n+1}=\begin{cases}
m_n+1&\text{if } x_{n+1}=\ell,\\
m_n-1&\text{if } a_n\le x_{n+1}<\ell,\\
m_n&\text{if } x_{n+1}>\ell \text{ or } x_{n+1}<a_n.
\end{cases}
$$
The first event has probability~$p_\ell$, while the second one has probability ${r_n:=p_{a_n}+\ldots+p_{\ell-1}\ge p_{\ell-1}}$.  If $p_{\ell-1}>p_\ell$, then the desired assertion is obvious. Otherwise, let $p_{\ell-1}=p_\ell=p$ and consider the random walk~$\{z_n\}$ on~$\mathbb Z$ with transition probabilities
$$
z_{n+1}=\begin{cases}
z_n+1&\text{with probability } p,\\
z_n-1&\text{with probability } p,\\
z_n&\text{with probability } 1-2p.
\end{cases}
$$
Now we use the well-known recurrence criterion for a random walk with step~$d$ (see, e.g.,~\cite{Spitzer}): it is recurrent if and only if $\lim\limits_{t\nearrow1}\int_{-\pi}^\pi\frac{dx}{1-t\phi(x)}=\infty$, where $\phi(x)={\mathbb E}e^{ixd}$. In our case, $\phi(x)=2p\cos x+1-2p$; it easily follows that the criterion is satisfied and the random walk is recurrent. Hence, by the properties of a recurrent random walk, the random walk~$\{z_n\}$ starting from~$m_q$ will reach~$0$ with probability~$1$. Now we apply coupling. Namely, consider the random process~ $\{z'_n\}_{n\ge q}$ on~$(X,m_p^\infty)$ defined as follows. Take a random variable~$\eps_n$ independent of all the other ones that is equal to~$1$ with probability~$\frac{p}{r_n}$ and~$0$ with probability~$1-\frac{p}{r_n}$. Set
$$
z'_{n+1}=\begin{cases}
z'_n+1&\text{if } x_{n+1}=\ell,\\
z'_n-1&\text{if } a_n\le x_{n+1}<\ell\text{ and } \eps_n=1,\\
z'_n&\text{otherwise}.
\end{cases}
$$
Clearly, on the one hand, $\{z'_n\}$ has the same distribution as~$\{z_n\}$ and, consequently, reaches~$0$ with probability~$1$. On the other hand, for every $n\ge q$ we have $m_n\le z_n'$. It follows that the original process~$\{m_n\}$ also reaches~$0$ with probability~$1$.
\end{proof}

The following lemma completes the proof of Theorem~\ref{th:part}.

\begin{lemma}\label{l:tails}
Two sequences $x,y\in X$  belong to the same element of the Young tail partition~$\bet_n$ if and only if their initial segments $[x]_n$~and~$[y]_n$ belong to the same plactic class and the tails
$\{x\}_{n+1}$~and~$\{y\}_{n+1}$ coincide.
\end{lemma}

\begin{proof}
We argue by induction on the number~$k$ of letters in the alphabet~$\A$. The base case $k=2$ is proved in Proposition~\ref{prop:young2}. We now prove the induction step $k-1\mapsto k$. Consider the subtableaux $P'([x]_i)$~and~$P'([y]_i)$  in $P([x]_i)$~and~$P([y]_i)$, respectively, consisting of all rows except the first one (and filled with $2,\ldots,k$). Then $\shape(P'([x]_i))=\shape(P'([y]_i))$ for $i\ge n$, hence, by the induction hypothesis,  $P'([x]_n)=P'([y]_n)$ and the sequences of elements bumped into the second row in $\{x\}_{n+1}$~and~$\{y\}_{n+1}$ coincide. We claim that
$m_n(k,x)= m_n(k,y)$ in the notation of Lemma~\ref{l:disappear}. Assume to the contrary that, say, $m_n(k,x)> m_n(k,y)$. Since the shapes of the growing tableaux coincide, it is clear that the difference
$m_i(k,x)-m_i(k,y)$ can decrease only if $k$~gets bumped from the first row of~$P([x]_i)$ and a smaller element gets bumped from the first row of~$P([y]_i)$, which, as noted above, cannot happen. However, it follows from Lemma~\ref{l:disappear} that there exists $j>n$ such that $m_j(k,x)=0$, a contradiction. Hence, $m_n(k,x)=m_n(k,y)$, and it follows from the above considerations that elements equal to~$k$ occupy the same positions  in $\{x\}_{n+1}$~and~$\{y\}_{n+1}$. Now note that these elements do not affect the growth of the subtableaux filled with the smaller elements. Denote by  $x'$~and~$y'$ the subsequences in $x$~and~$y$, respectively, obtained by discarding the elements equal to~$k$. It follows from what we have proved that
$x'\sim_{\bet_{n'}}y'$, where $n'$~is the number of elements less than~$k$ in $[x]_n$~and~$[y]_n$. It remains to apply the induction hypothesis to $x'$~and~$y'$.
\end{proof}

\begin{corollary}
$$
\bxi_n\prec\si_n,\qquad \bet_n\succ\zeta_n\succ \tau_n.
$$
\end{corollary}

Now we turn to the proof of Theorem~{\rm\ref{th:main}.

1. If $x\sim_{\bxi}y$, then $\shape([x]_n)=\shape([y]_n)$ for all~$n$, and it follows from Lemma~\ref{l:tails} with $n=0$ that $x=y$.

2. As in Proposition~\ref{prop:k2}, we want to use the de Finetti partitions and the Hewitt--Savage law. Namely, the desired result follows by the Hewitt--Savage law from the following lemma.

\begin{lemma}\label{l:definetti}
If $x\sim_{\zeta_n}y$, then there exists $N\ge n$ such that $x\sim_{\bet_N}y$.
\end{lemma}

\begin{proof}
Since $\zeta_n$ is the orbit partition for an action of the symmetric group~$\sn$, it suffices to prove the assertion in the case where $x$~and~$y$ are obtained from each other by the action of a Coxeter generator
$\si_i=(i,i+1)$, i.e., by a transposition of $x_i$~and~$x_{i+1}$. Assume without loss of generality that
$x_i=u<v=x_{i+1}$. Then $y_i=v$, $y_{i+1}=u$, and $y_j=x_j$ for $j\ne i,i+1$.

Denote by $R^{(j)}(x)$ and $R^{(j)}(y)$ the first rows of the tableaux~$P_j(x)$ and~$P_j(y)$, respectively (as multisets). We claim that almost surely there exists~$N\ge i+1$ such that $R^{(N)}(x)=R^{(N)}(y)$.
If $R^{(i+1)}(x)= R^{(i+1)}(y)$, there is nothing to prove. Otherwise, $\max R^{(i)}(x)=u$.
Set $v_{i+1}:=v$.
Then $R^{(i+1)}(x)=R^{(i+1)}(y)\cup\{v_{i+1}\}$.

Assume that at the $j$th step
\begin{equation}\label{+1}
R^{(j)}(x)=R^{(j)}(y)\cup\{v_{j}\}.
\end{equation}
Set $A_j:=\{d\in R^{(j)}(x):d<v_j\}$ and ${B_j:=\{d\in R^{(j)}(x):d\ge v_j\}\setminus\{v_j\}}$ (multisets) and denote $u_j:=\max A_j$. In particular, $u_{i+1}=u$ and $B_{i+1}=\emptyset$. Look at the insertion of an element~$x_{j+1}$ with $j>i$. Clearly, if $x_{j+1}< u_j$ or $x_{j+1}\ge v_j$, then $R^{(j)}(x)$~and~$R^{(j)}(y)$
undergo the same changes; in this case, we set $v_{j+1}=v_j$, and~\eqref{+1} remains valid.

If  ${u_j
\le x_{j+1}<v_j}$, then
$R^{(j+1)}(x)=R^{(j)}(x)\setminus\{v_j\}\cup\{x_{j+1}\}$ and two cases are possible. If $B_j\ne\emptyset$, then ${R^{(j+1)}(y)=R^{(j)}(y)\setminus\{\min B_j\}\cup\{x_{j+1}\}}$, and~\eqref{+1} remains valid with $v_{j+1}=\min B_j$. Finally, if ${B_j=\emptyset}$, then $R^{(j+1)}(y)=R^{(j)}(y)\cup\{x_{j+1}\}$ and $R^{(j)}(x)=R^{(j)}(y)$.

We claim that this will eventually happen with probability~$1$. Assume the contrary. Note that
$v_j$~never decreases and can increase only finitely many times, because the alphabet~$\A$ is finite. Let $v_j=v$ for all sufficiently large~$j$. By Lemma~\ref{l:disappear}, almost surely there are infinitely many~$j$ such that  $m_j(v,y)=0$, i.e., $v\notin B_j$.  Hence, almost surely one of them is succeeded by the event  $u_j\le x_{j+1}<v$ (which has probability $\ge p_{v-1}>0$). If at this moment $B_j\ne\emptyset$, then $v_{j+1}=\min B_j>v$, a contradiction. Therefore,
$B_j=\emptyset$ and, as shown earlier,  $R^{(j)}(x)=R^{(j)}(y)$.

So, we have proved that if $x\sim_{\zeta_n}y$, then with probability~$1$ there exists~$N$ such that $P_N(x)$~and~$P_N(y)$ have the same first row. But then the sequences of elements bumped into the second row in these tableaux also differ only by a permutation, hence, we obtain by induction that all rows (there are finitely many of them) eventually become equal. The lemma is proved.
\end{proof}

As we have discussed earlier, the second assertion of Theorem~{\rm\ref{th:main}} follows from Lemma~\ref{l:definetti} by the Hewitt--Savage zero--one law. The theorem is proved.

\end{document}